\documentclass[12pt]{amsart}
\usepackage{geometry}
\usepackage{amsmath, amssymb}
\usepackage{amscd}
\usepackage{enumitem}
\usepackage{comment}
\usepackage[english]{babel}
\usepackage{graphicx}
\usepackage{epstopdf}
\numberwithin{equation}{section}

\newtheorem{thm}{Theorem}[section]
\newtheorem{lem}[thm]{Lemma}

\theoremstyle{definition}
\newtheorem{defn}[thm]{Definition}

\theoremstyle{remark}
\newtheorem{rem}[thm]{Remark}

\begin{document}


\centerline{\bf Balanced finite presentations of the trivial group.}

\centerline{\bf Boris Lishak}
\centerline{\bf Department of Mathematics, University of Toronto}
\centerline{\bf 40 St. George st., Toronto, ON, M5S 2E4, CANADA;}
\centerline{\bf bors@math.toronto.edu}
\vskip 1truecm

{\bf Abstract.} We construct a sequence of balanced finite presentations
of the trivial group with two generators and two relators with the following property: The minimal number of relations required to demonstrate that
a generator represents the trivial element grows faster than the tower
of exponentials of any fixed height of the length of the finite presentation.

\section{Introduction} \label{intro}

The purpose of this paper is to construct a sequence of balanced
finite presentations of the trivial group with two generators
and two relators, one of which is the same for all finite presentations
in the sequence, with the following property: the minimal number of 
Tietze transformations required to bring these finite presentations
to the empty presentation of the trivial group grows faster
than the tower of exponetials of any fixed height of the length of
the variable relation (or, equivalently, the length of these finite
presentations). The minimal area of the van Kampen diagram required
to demonstrate that either of the generators is trivial also grows
faster than the tower of exponentials of any fixed height. 

Balanced finite presentations can be realized as ``obvious" finite
presentations of the (trivial) fundamental group of $4$-dimensional
spheres and discs. This fact leads to numerous geometric applications
of our results that will be discussed in my joint paper
with A. Nabutovsky \cite{LishNab}.
However, they also provide the following group-theoretic implications.

Recall, the Magnus problem (cf. \cite{nb}) asks whether or not the triviality
problem for balanced group presentrations is algorithmically solvable.
Equivalently, it asks whether or not the size of van Kampen diagrams
required to show that all generators in a balanced finite presentation
of the trivial group are trivial cannot be majorized by any
computable function of the complexity of the finite presentation.
Although we are not able to prove that every computable function
can serve as a lower bound for the size of such van Kampen
diagrams for all sufficientlty large values of the
complexity of the finite presentations, we establish that it already grows
very fast for balanced finite presentations with two generators
and two relators.

Another implication of this work is that one can have balanced finite
presentations of the trivial group satisfying the (balanced) Andrews-Curtis
conjecture where one needs an enormous number of elementary Tietze
moves to transform a given finite presentation to the trivial
finite presentation of the trivial group. This fact is relevant in light
of recent work (\cite{Mias}, \cite{HR}), where it was verified by means of explicit
computations that certain
specific balanced finite presentations of the trivial group cannot be
transformed into the trivial one by means of not too many elementary
Tietze operations. Our result casts some doubt on whether these results can be
considered as a strong empirical evidence that these
balanced finite presentations are counterexamples to
the Andrews-Curtis conjecture. Also, they can potentially help to exclude approaches to proving this conjecture that would result in estimates that are
not very rapidly growing.

In order to prove our results we start from the Baumslag-Gersten group,
that has Dehn function that is not bounded by any tower of exponents of a finite
length (\cite{Gers1}, \cite{Gers2}, \cite{Plat}). This group has finite
presentation $<x, y| x^{x^t}=x^2>$, and there exist trivial words $w_n$ 
of length $O(2^n)$  such that the area of any van Kampen diagram
grows as a tower of exponents of height $n$. The proof
of a lower bound for the Dehn function given by Gersten uses the
fact that this group can be obtained from
a cyclic group by a sequence of two HNN-extension and, therefore,
is aspherical. Therefore, the universal covering of its representation
complex is contractible, has trivial second homology group,
and each filling of each trivial word is unique on $2$-chain
level (see survey papers \cite{Sapir}, \cite{Bridson1}, \cite{Short} for discussions of the Dehn functions and alternative proofs for the lower bounds). A natural idea 
is to kill this group using new relation $w_n=t$. However,
as the resulting group is trivial, we  cannot use covering spaces of the realization
complex, and direct combinatorial proofs of desired lower bound
for the areas of van Kampen diagrams for generators of these
finite presentations seem elusive.

I learned about this problem from my Ph.D. advisor Alexander Nabutovsky.
He usuccessfully tried to prove that when one kills the Baumslag-Gersten group by adding an extra  relation such as $t=w_n$, one obtains a sequence of desired ``complicated" balanced finite presentations of the trivial group. He discussed this problem with several mathematicians in 1994-2000 but no solution was found.

After the first version of this paper was posted on arXiv I recieved an e-mail from Martin Bridson. He informed me that he had figured out how to use the Baumslag-Gersten group to construct a sequence of balanced presentations of the trivial group that require huge numbers of Andrews-Curtis moves to trivialize. He announced this result in 2003 at the Spring Lecture Series in Arkansas and in several other talks. In his talks he also observed that the existence of such examples make computer experiments suggesting that certain presentations are potential counterexamples to the Andrews-Curtis conjecure not convincing. This result was mentioned in his ICM-2006 talk (\cite{Bridson2}, page 977). However, he did not publish or post on the Internet any detais of his construction, and therefore I was not aware of his work during the time that I worked on this paper. He e-mailed me a preprint with his results (\cite{Bridson3}), after I posted my paper, and this has now appeared on the arXiv. His constructions and methods are different from ours and, in particular, they do not produce examples of rank $2$.

Our approach is to kill the Baumslag-Gersten group using a longer
(variable) second relator, so as to be able to
use a version of the small
cancellation theory to prove the desired property of van Kampen diagrams.
We use a combination of the small cancellation theory for HNN-extension deloped by  G. Sacerdote and P. Schupp  in \cite {SacerSch} (see also \cite{LS}) and ideas of A. Olshanskii related with the concept
of ``contiguity subdiagrams" (\cite{Olsh1},\cite{Olsh2}). Finally, we modify this theory
to introduce a concept of equivalence between words based not
on their equality in a group but its quantitative version, namely, the equality that can be established by means of a van Kampen diagram of area that does not exceed
a specified number $N$.

\section{Main Results} \label{main}

In this section we construct a sequence of balanced group presentations of the trivial group with the following properties.
Loosely speaking, the presentations are simple but to transform them to the empty presentation requires to go through increasingly more complex presentations.
To make this notion precise we use elementary Tietze transformations (cf. \cite{BHP}). We denote a presentation $\mu$ by $\mu = ( \{x_1,...,x_r\} , \{a_1,..., a_p \} )$, where $x_1,...,x_r$ are letters, $a_1,...,a_p$ are words in $x_j^{\pm 1}$, and the group presented by $\mu$ is $<x_1,...,x_r|a_1=1,..., a_p=1>$.

\begin{defn}
Elementary Tietze transformations:\\

${Op}_1$: The presentation $\mu = ( \{x_1,...,x_r\} , \{a_1,...,a_{i-1}, a_i \equiv a'a'', a_{i+1}, ..., a_p \} )$, where $a_1,...,a_p$ are words in the $x_j^{\pm 1}$, is replaced by the presentation $\mu_1 = ( \{x_1,...,x_r\} ,$  $ \{a_1,...,a_{i-1}, a' x_j^{\epsilon} x_j^{-\epsilon} a'', a_{i+1}, ..., a_p \} )$ ($\epsilon = \pm 1$).\\

${Op}_1^{-1}$: The inverse of ${Op}_1$ - it deletes $ x_j^{\epsilon} x_j^{-\epsilon}$in one of the relators.\\

${Op}_2$: The presentation $\mu = ( \{x_1,...,x_r\} , \{a_1,...,a_{i-1}, a_i, a_{i+1}, ..., a_p \} )$ is replaced by the presentation $\mu_1 = ( \{x_1,...,x_r\} , \{a_1,...,a_{i-1}, a'_i, a_{i+1}, ..., a_p \} )$, where the word $a'_i$ is a cyclic permutation of the word $a_i$.\\

${Op}_3$: The presentation $\mu = ( \{x_1,...,x_r\} , \{a_1,...,a_{i-1}, a_i, a_{i+1}, ..., a_p \} )$ is replaced by the presentation $\mu_1 = ( \{x_1,...,x_r\} , \{a_1,...,a_{i-1}, a_i^{-1}, a_{i+1}, ..., a_p \} )$.\\

${Op}_4$: The presentation $\mu = ( \{x_1,...,x_r\} , \{a_1,...,a_{i-1}, a_i, a_{i+1}, ..., a_p \} )$ is replaced by the presentation $\mu_1 = ( \{x_1,...,x_r\} , \{a_1,...,a_{i-1}, a_ia_j, a_{i+1}, ..., a_p \} )$, where $i \neq j$.\\

${Op}_5$: The presentation $\mu = ( \{x_1,...,x_r\} , \{a_1, ..., a_p \} )$ is replaced by the presentation $\mu_1 = ( \{x_1,...,x_r, x_{r+1}\} , \{a_1, ..., a_p, x_{r+1} \} )$ .\\

${Op}_5^{-1}$: The inverse of ${Op}_5$.

\end{defn}

It is well-known (and can be found, for example, in \cite{BHP}) that any presentation of the trivial group could be reduced to the empty presentation by a sequence of Tietze transformations if we can also add empty relators to the presentation. The stable Andrews-Curtis conjecture says that we do not need to add empty relators for balanced presentations. Note that in \cite{BHP} slightly different ${Op}_5^{\pm 1}$ are used. There one can add a relator $x_{r+1} a $, where $a$ is a word in the letters $x_1^{\pm 1},..., x_r^{\pm 1}$ . But that is not necessary for presentaions of the trivial group.\\
\\

We introduce some notations. Denote by $E_n$ the tower of exponents of height $n$, i.e. $E_i$ are recursively defined by $E_0 = 1$, $E_{n+1} = 2^{E_n}$. As usual, $x^y$ denotes $y^{-1}xy$, where $x, y$ can be words or group elements. Let $l(w)$ be the length of the word $w$. If $w$ represents the identity element, denote by $Area_{\mu}(w)$ the minimal number of 2-cells in a van Kampen diagram over the presentation $\mu$ with boundary cycle labeled by $w$. Now we can state the main theorem:

\begin{thm} \label{main}
There exist presentations of the trivial group $\mu_i = $ \\ $ (\{x,y,t\},\{x^{y}x^{-2}, x^ty^{-1}, a_i \})$ for $i \in \{5,6,7,... \}$, where $l(a_i) < 100 \cdot 2^{i}$, but the minimal number of elementary Tietze transformations required to bring $\mu_i$ to the empty presentation is at least $E_{i-2}$.
\end{thm}

First, we give an outline of the proof. Notice that $\mu_i$ without the last relator $a_i$ (denoted $\mu_0 = (\{x,y,t\},\{x^{y}x^{-2}, x^ty^{-1} \})$) presents the Baumslag-Gersten group $G = <x,y,t | y^{-1}xy=x^2, y = t^{-1}xt>$ (the base group for this HNN extension is called Baumslag-Solitar). It is known that the Dehn function for $G$ is $E_{log_{2}(n)}$ (see \cite{Plat}, \cite{Gers2}). In particular they produce words $w_n$ representing the identity element of length less than $16 \cdot 2^n$ but of area greater than $E_n$. If we add the relation $t = w_n$ to $G$, then the group collapses to the trivial group. Actually $a_n$ will be slightly more compicated than just $t^{-1}w_n$, we will define it precicely later, but in any case it will kill $t$ by equating $t$ to a word of large area. Then we will prove the following theorem:

\begin{thm} \label{second}
The presentations of the trivial group $\mu_i = (\{x,y,t\},\{x^{y}x^{-2}, x^ty^{-1}, a_i \})$ for $i \in \{5,6,7,... \}$, where $l(a_i) < 100 \cdot 2^{i}$, have the property that $Area_{\mu_i}(x)>E_{i-1}$.
\end{thm}

Theorem \ref{main} follows easily from Theorem \ref{second}, we will supply all the details in the end of the paper. Now, we give a brief outline of the proof of Theorem \ref{second}.

What makes the theorem difficult is that though $Area_{\mu_0}(w_n)$ is large, this fact a priori doesn't give us any bounds on $Area_{\mu_n}(w_n)$. Furthermore, standard techniques for proving lower bounds for the area do not work for presentations of the trivial group. Unlike an HNN-extension $G$ the trivial group has no structure, no normal form theorem, to use to prove that different $2$-cells in a van Kampen diagram cannot cancel each other to form a smaller van Kampen diagram. Therefore we will proceed as follows.

The authors of \cite{SacerSch} have developed the small cancellation theory over HNN extensions: given a group $H$, an HNN extension, and a new group formed from $H$ by adding new relations, one can check if the new relations satisfy small cancellation condition (over $H$) (see also the exposition of this theory in \cite{LS}). The length of the pieces is defined as the number of occurences of the stable letter of a normal form of the relator and its cyclic permutations. The relators $a_i$ will not be in a normal form (i.e. $t$-reduced), as the normal form of the identity element has length zero (no $t$ letters), but it will be reduced in a weaker sense. If we only allow a limited number of applications of the relation $x^ty^{-1}$, $a_i$ will be reduced and will satisfy the small cancellation conditions. Then by the small cancellation theory $x$ is not trivial in $\mu_n$ (given the limited number of applications of $x^ty^{-1}$). To develop this "effective" small cancellation theory one can use the same language of van Kampen diagrams over an HNN extension as in \cite{LS}, but that results in complicated tallying of the number of applications of $x^ty^{-1}$ in this or that $t$-reduction. Therefore, we decided to use the techniques of contiguity subdiagrams used for stratified small cancellation theory \cite{Olsh1}, small cancellation theory over hyperbolic groups \cite{Olsh2}, relatively hyperbolic groups \cite{Osin}. We will first define $a_i$ so that the reader has a motivating example, then develop small cancellation theory over HNN extensions with a limited number of applications of relations, and then prove Theorem \ref{second}.\\

Let $a_n=t^{-1} u_n$, where $u_n$ is defined inductively as follows. Let 

\begin{displaymath}
u_{n,0} = [y^{-E_n}xy^{E_n},x^{3}]  [y^{-E_n}xy^{E_n},x^{5}]  [y^{-E_n}xy^{E_n},x^{7}] .
\end{displaymath}
Suppose $u_{n,m}$ is defined, then let $u_{n,m+1}$ be the word obtained from $u_{n,m}$ by replacing subwords $y^{\pm E_{n-m}}$ with $t^{-1} y^{-E_{n-m-1}}x^{\pm 1}y^{E_{n-m-1}} t$. Finally, let $u_n = u_{n,n}$. 

\begin{rem} \label{relator}
Since in the inductive step we replace subwords by equivalent (in $G$) subwords, $u_n =_G u_{n,n} =_G u_{n,n-1} =_G ... =_G u_{n,0}$. Each commutator in $u_{n,0}$ is the identity element in $G$, which makes $u_n =_G 1 $. Note that if we take $w_{n,0} = [y^{-E_n}xy^{E_n},x^1]$ and apply the same inductive procedure we will get $w_n$, the word from \cite{Plat} of area at least $E_n$ (see Figure \ref{fig:wn}). The reason for us to make $u_n$ slightly different is to satisfy a small canellation condition. $x$ to the powers $\pm 3, \pm 5, \pm 7$ act as a unique signature among cyclic permutations of $u_n$. It will be clear later why we avoid powers of $2$. We can make an estimate $l(a_n) \le 100 \cdot 2^n$.
\end{rem}

\begin{figure}[h]
\centering
\includegraphics[scale=0.4]{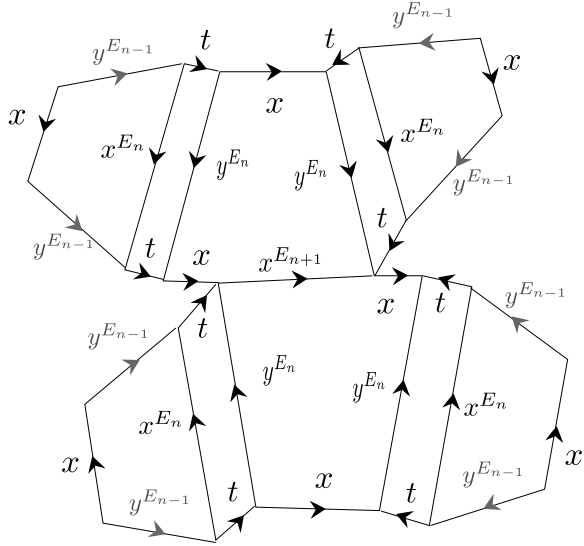}
\caption{Van Kampen diagram for $w_{n,1}$. Parts of the boundary to be replaced by shorter paths to get $w_{n,2}$ are marked by grey.}
\label{fig:wn}
\end{figure}

\section{Proofs} \label{proofs}

Now we develop small cancellation theory over HNN extensions with the limited number of operations with the stable letter. We will be using diagram approach to HNN extensions first used in \cite {MillerSchupp}, see also an exposition in \cite{Sapir}, and in \cite{Short} for dual diagrams. The main instrument of this approach is the $t$-band. Given a presentation of an HNN extension $H$ with the stable letter $t$, a $t$-band is a collection of $t$-cells in a van Kampen diagram over this presentation such that the cells are adjacent to each other along $t$-edges. We will be using implicitly the following fact. The outer edge of a $t$-band annulus is a trivial in $H$ element and therefore trivial in the base group by Britton's lemma (cf \cite{LS}). Therefore the subdiagram bounded by the outer edge can be replaced by a diagram containing cells from the base group only. (In fact, one can see why Britton's lemma is true by replacing innermost annuli using the definition of HNN extensions: commuting with $t$ applies the isomorphism of the subgroups, see \cite{Sapir}, \cite{Short} for details.) We need the following analogue of "reduced" in HNN extensions.

\begin{defn}
Let $\mu_H$ be a presentation of an HNN extension $H$ with the stable letter $t$ and associated subgroups $A,B$ ($t^{-1}At = B$). Let $N$ be a natural number. We call a word $w=g_0 t^{\delta_1} g_1 ... g_{m-1} t^{\delta_m} g_m$ ($g^i$ are words in the letters of the base group, $\delta_i = \pm 1$) $N$-reduced if the following holds:\\

If $\delta_i = -1$ and $\delta_{i+1} = 1$ then either $g_i$ is not in $A$, or $t^{-1}g_i t = g \in B$ but any van Kampen diagaram for $t^{-1}g_i t g^{-1}$ contains more than $N-1$ $t$-cells.\\

If $\delta_i = 1$ and $\delta_{i+1} = -1$ then either $g_i$ is not in $B$, or $t g_i t^{-1} = g \in A$ but any van Kampen diagaram for $tg_i t^{-1} g^{-1}$ contains more than $N-1$ $t$-cells. 
\end{defn}

\begin {rem} \label {reduced}
I.e. pinches are allowed only if they are witnessed by long enough $t$-bands; $\infty$-reduced is the usual notion of "reduced" for HNN extensions. Note that this definition is independent of the particular words $g_i$ as long as they represent the same elements of the base group. Therefore we think of reduced words as sequences $g_0, t^{\delta_1}, g_1, ... g_{m-1}, t^{\delta_m}, g_m$, where $g_i$ are elements of the base group. However, two different reduced sequences can reprsent the same element of $H$: unlike normal forms, where coset representatives are fixed (see \cite{LS} for details), reduced forms have this ambiguity. We write $w_1 \equiv w_2$ if they are the same as sequences of letters $t^{\pm 1}$ and elements of the base group.
\end {rem}
Similarly, we introduce cyclically $N$-reduced:

\begin{defn} \label{cycl}
Let $\mu_H$ be a presentation of an HNN extension $H$. Let $N$ be a natural number. We call a word $w=g_0 t^{\delta_1} g_1 ... g_{m-1} t^{\delta_m} g_m$ ($g^i$ are in the base group, $\delta_i = \pm 1$) cyclically $N$-reduced if all cyclic permutations of $g_0, t^{\delta_1}, g_1, ..., g_{m-1}, t^{\delta_m}, g_m$ are $N$-reduced.
\end{defn}

One can now prove the $N$-reduced versions of Britton's Lemma and Collins' Lemma \cite{LS}, but since we have chosen not to consider van Kampen diagrams over $H$ we don't need them: their meaning will be incorporated in the proof of our main small cancellation result (Theorem \ref{small}). We pove a lemma about $G$:
\begin{lem} \label{baum}
Let $v_1 = t^{-1} g_1 t$, $v_2 = t g_2 t^{-1}$, where $g_1, g_2$ are words in the letters $x,y,x^{-1},y^{-1}$ and $g_1=x^{i}$, $g_2=y^{i}$ in the Baumslag-Solitar group. Then $v_1$, $v_2$ are $i$-reduced (relative to $\mu_0$).
\end{lem}

\begin{proof}
We prove this lemma for $v_1$, for $v_2$ it's completely analogous. Consider a van Kampen diagram for $t^{-1} g_1 t g^{-1}$, where $g = y^i$ (see Figure \ref{fig:baum}). Since there are only two letters $t$ on the boundary of the diagram, they must be connected by a $t$-band, say of length $k$. Because $G$ is an HNN extension of the Baumslag-Solitar group, the latter embeds in $G$. Therefore $i=k$.
\end{proof}

\begin{figure}[h]
\centering
\includegraphics[scale=0.4]{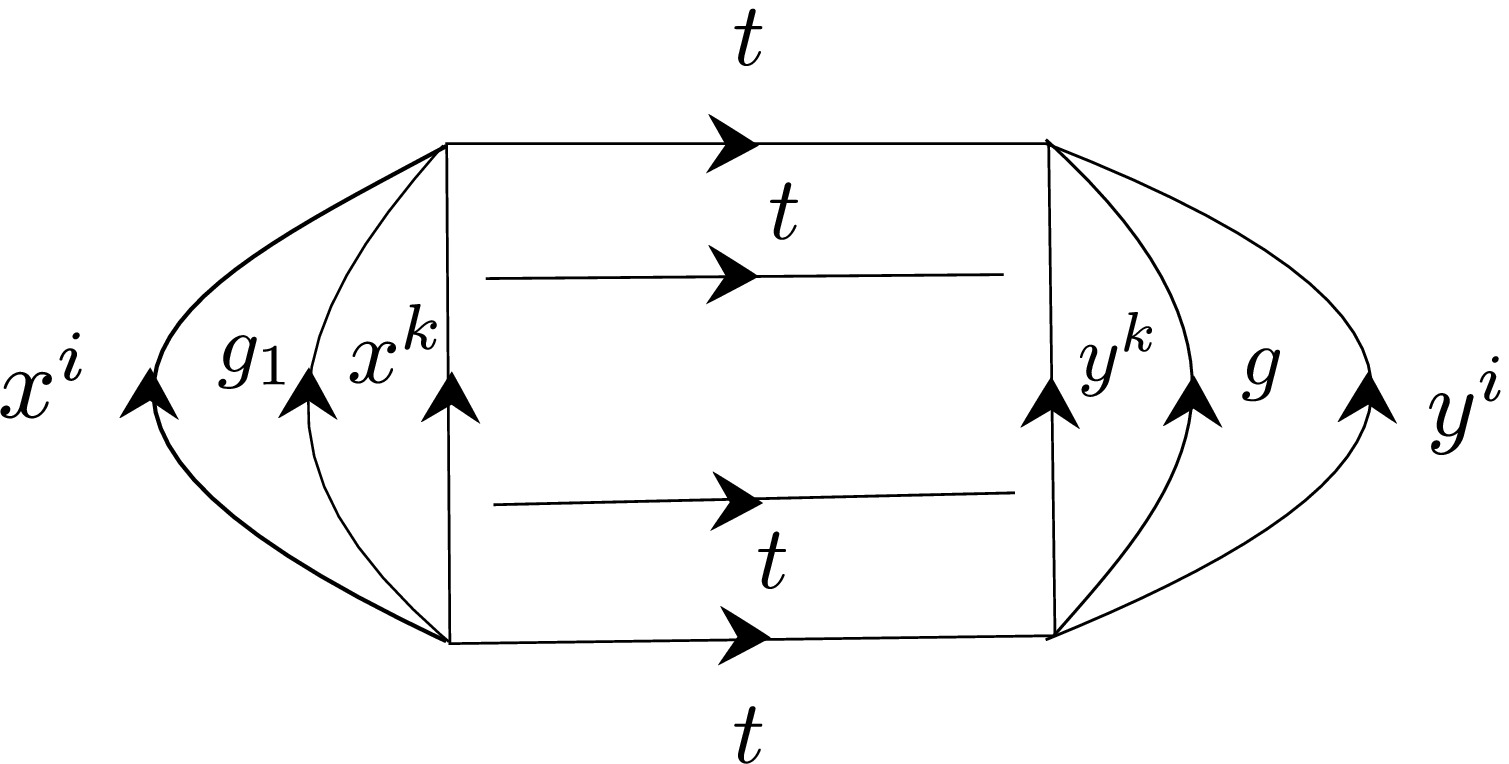}
\caption{}
\label{fig:baum}
\end{figure}

\begin{rem} \label{relator2}
Notice that 

\begin{displaymath}
u_{n,1} = [t^{-1} y^{-E_{n-1}}x^{-1}y^{E_{n-1}} t x t^{-1} y^{-E_{n-1}}x^{1}y^{E_{n-1}} t,x^{3}] ... 
\end{displaymath}
is $E_n$-reduced by the preceding lemma.
(Here we omitted the other two commutators involving $x^5$ and $x^7$ instead of $x^3$ , which look the same.)
Similarly, $u_{n,2}$ is $E_{n-1}$-reduced, $u_{n,3}$ is $E_{n-2}$-reduced, etc... Also, $u_{n,1}$ is cyclically $E_n$-reduced, and so is $t^{-1}u_{n,1}$ and $(t^{-1}u_{n,1})^{-1}$. Since our theory will only work with cyclically $N$-reduced words, it is a problem that $u_n$, $a_n$ are not cyclicaly $N$-reduced for large $N$. We will deal with this in the proof of Theorem \ref{second}, but for now we prove results about $u_{n,1}$.

\end{rem}

We will state the small cancellation condition we need. First we will define a piece.

\begin{defn} \label{piece}
Let $\mu_H$ be a presentation of an HNN extension $H$ with the stable letter $t$ and associated subgroups $A,B$ ($t^{-1}At = B$). Let $N$ be a natural number. Let $R$ be a set of cyclically $N$-reduced words in $\mu_H$, such that $R$ is closed under inversion and cyclic permutation of the reduced sequences  (see Remark \ref{reduced}). We call such a set $R$ $N$-symmetrized. Let $r \equiv pb$, $r' \equiv p'b'$ be in $R$. We call $p$ an $N$-piece if the following holds:\\
(1) both $p$ and $p'$ start and end with $t^{\pm 1}$, \\
(2) $p = v_1 p' v_2$ in $H$, \\
(3) $v_1, v_2$ are in the base group,\\
(4) there exists a van Kampen diagram for $p = v_1 p' v_2$, where all $t$-bands are of length less than $N$,\\
(5) $b \not \equiv v_1 b' v_2$.\\
See Figure \ref{fig:piece}.
\end{defn}

\begin{rem}
Like in the usual small cancellation theory condition (5) is to make sure that if $r=r'$ we do not get pieces of the length of $r$. We can get away with this because we can cancel the corresponding cells in the van Kampen diagram making it reduced (see the proof of Theorem \ref{small} for details). If we make the condition (5) even stronger:  $b \neq_H v_1 b' v_2$, and if we also forget (4) we can get the small cancellation results over HNN extensions of \cite{LS}. We require a weaker condition (5) because otherwise in the cancelling of the coresponding cells in the van Kampen diagram we might need to add $t$-cells, which we do not want. With (5) as it is we might be required to add cells from the base group of the HNN extension only. Our $N$-symmetrized set will be such that we can afford the weaker condition (5) (see Lemma \ref{c6}).
\end{rem}

Denote by $l_t(w)$ the number of occurences of the letters $t, t^{-1}$ in the word $w$. We will measure the length of pieces with $l_t$. It makes sense, because for two $N$-reduced words $l_t$ the same if the words are equal in $H$ (given our usual restriction on the length of $t$-bands). We define a metric small cancellation condition:

\begin{defn}
Let $R$ be an $N$-symmetrized set (as in the previous definition).\\
We say condition C'($\lambda$,$N$) is satisfied for $R$ if $r\in R$, $r \equiv pb$, where $b$ is an $N$-piece, implies $l_t(p) < \lambda l_t(r)$.
\end{defn}

We will only prove one small cancellation result, the only one we need:

\begin{thm} \label{small}
Let $\mu_H$ be a presentation of an HNN extension $H$ with the stable letter $t$ and associated subgroups $A,B$ ($t^{-1}At = B$). Suppose we have words $r_1,...,r_m, w$, such that $r_1,...,r_m$ generate an $N$-symmetrized set, $w \neq 0$ in $H$ and $l_t(w)$=0. Denote by $\mu$ the presentation obtained from $\mu_H$ by adding $r_1,...,r_m$ as relators. Then if $R$ satisfies C'($\frac{1}{6}$,$N$) there is no van Kampen diagram over $\mu$ with boundary cycle $w$ which contains less than $N$ $t$-cells ($t$-cells from $\mu_H$).
\end{thm}
\begin{proof}
By $r$-cells we will call cells corresponding to the relators $r_1,...,r_m$. \\

Suppose for the sake of contradiction that we have such a diagram $D$ which has less than $N$ $t$-cells. We can assume this diagram is reduced. We also want it to be $r$-reduced it in the following sense. We choose $D$ to have minimal number of $r$-cells granted we do not increase the number of $t$-cells.\\

Now consider $t$-bands in the diagram $D$. Since the boundary does not have any letters $t$, $t$-bands have to either be rings or orginate and end at the $r$-cells. We call two $t$-bands consecutive if they begin and end on consecutive letters $t^{\pm 1}$ of the same $r$-relator (by this we mean they can be separated by letters other than $t^{\pm 1}$) and the subdiagram bounded by this two $t$-bands on the sides and two $r$-cells at the ends contains only $\mu_H$ cells.\\

A maximal sequence of $t$-bands, including the cells in between, such that neigbouring bands are consecutive, will be called a {\it $t$-cable} (see Figure \ref{fig:cable}). We claim that the ends of a $t$-cable are $N$-pieces. All requirements of Definition \ref{piece} are straightforward to check, we only note that requirement (5) follows from the fact that $D$ is $r$-reduced. If (5) is not satisfied we can replace the subdiagram bounded by $b^{-1} v_1 b' v_2$ (in the notation of the definition \ref{piece}) by a diagram containing only cells from the base of the HNN extension. \\

\begin{figure}[h]
\centering
\includegraphics[scale=0.5]{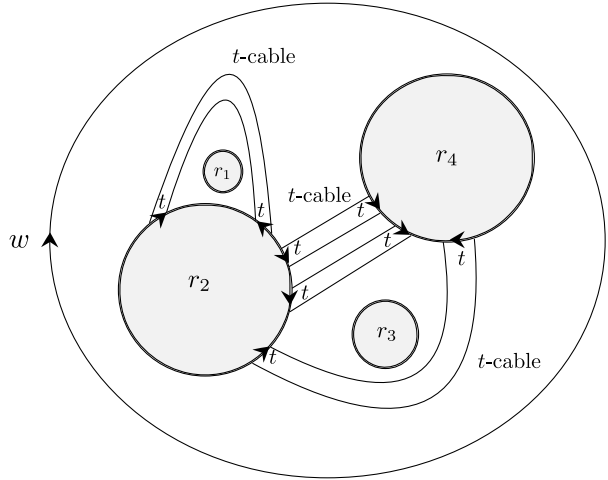}
\caption{An example of $D$. $r$-cells are marked with grey. There are four $t$-bands and three $t$-cables in this example. Note that this example is not realistic, because $r_1$, $r_3$ have no letters $t$.}
\label{fig:cable}
\end{figure}

We define a dual diagram to $D$, call it $D^*$. Associate to each $r$-cell a point. Connect them by edges corresponding to $t$-cables. Take an innermost connceted component of this graph and call it $D^{*}$. We claim that $D^{*}$ does not have any faces with less than $3$ edges. The fact that $D^*$ is an innermost component implies that the faces of $D^*$ are faces of the whole dual to $D$.\\

A face with $1$ edge ($L_1$ on Figure \ref{fig:dual}) would mean that there is a $t$-cables with both ends on the same cell, such that there are no other $t$-cables in between that $t$-cables and the relator (call that subdiagram $D'$). That implies that there are no letters $t^{\pm 1}$ on the boundary of that relator between the pieces (marked as $g_1$ on the figure) and there are no other $r$-cells in $D'$, but that contradicts that all cyclic permutations of $r_1,...,r_m$ are $N$-reduced.\\

A face with $2$ edges (marked $L_2$, $L_3$ on Figure \ref{fig:dual}) would arise from two $t$-cables between two (possibly one) cells such that there are no other $r$-cells in between. But that would contradict that $t$-cables are maximal sequences of consequitive $t$-bands.\\

\begin{figure}[h]
\centering
\includegraphics[scale=0.5]{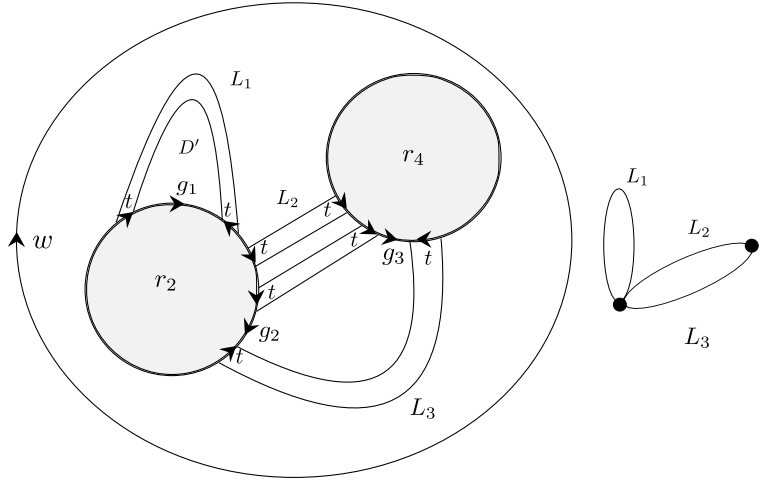}
\caption{An (impossible) example of $D$ and $D^*$ next to it. $L_1$ can not be a $t$-cable and $L_2$, $L_3$ should be one $t$-cable.}
\label{fig:dual}
\end{figure}

$D^*$ is a planar connected graph, therefore we have the Euler's Formula

\begin{displaymath}
V-E+F=1.
\end{displaymath}
From the fact that each face of $D^*$ has at least $3$ edges we have $F \leq \frac{2E}{3}$. Euler's Formula then becomes

\begin{displaymath}
1=V-E+F \leq V - \frac{E}{3}.
\end{displaymath}

Since the ends of a $t$-cable correspond to $N$-pieces, from C'($\frac{1}{6}$,$N$) we know that each vertex has degree at least $6$ (if we count a looping edge twice). Therefore we have $E \geq \frac{6V}{2} = 3V$, and the inequality becomes

\begin{displaymath}
1 \leq V - \frac{E}{3} \leq V - V = 0.
\end{displaymath}

We have reached a contradiction.

\end{proof}

To apply the preceding theorem we need a C'($\lambda$,$N$) $N$-symmetrized set. We will first need the following lemma.

\begin{lem} \label{solitar}
In the Baumslag-Solitar group (the base group of $G$) if $y^i x^m y^j = x^k$ then $i=-j$ and $\frac{m}{k} = 2^{j}$.
\end {lem}

\begin {proof}

Since the right part of the given equation does not contain any letter $y^{\pm 1}$, $i=-j$. If $i<0$,  $y^i x^m y^j = x^{2^i m} y^i y^j = x^{2^i m}$, and the result follows. If $i>0$, then by conjugating both sides of the equation we have  $x^m = y^{-i} x^k y^{-j} = x^{2^{-i} k}$ as in the previous case.

\end {proof}

Now we can prove

\begin{lem} \label{c6}

The set of all cyclic permutations of $t^{-1}u_{n,1}$ and $(t^{-1}u_{n,1})^{-1}$ is an $E_n$-symmetrized set satisfying condition C'($\frac{1}{6}$,$E_n$) over $G$.

\end {lem}

\begin {proof}
The set is $E_n$ symmetrized by Remark \ref{relator2}. If $p$ is a piece corresponding to $p'$ there is a van Kampen diagram corresponding to $p = v_1 p' v_2$ (from the definition of the piece). Note that this diagram is like a $t$-cable (see proof of Theorem \ref{small}). Each element of the Baumslag-Solitar $g_i$ of $p$ is "carried" by the neigbouring $t$-bandes to the corresponding element $g'_i$ of $p'$ (see Figure \ref{fig:piece}).\\

\begin{figure}[h]
\centering
\includegraphics[scale=0.5]{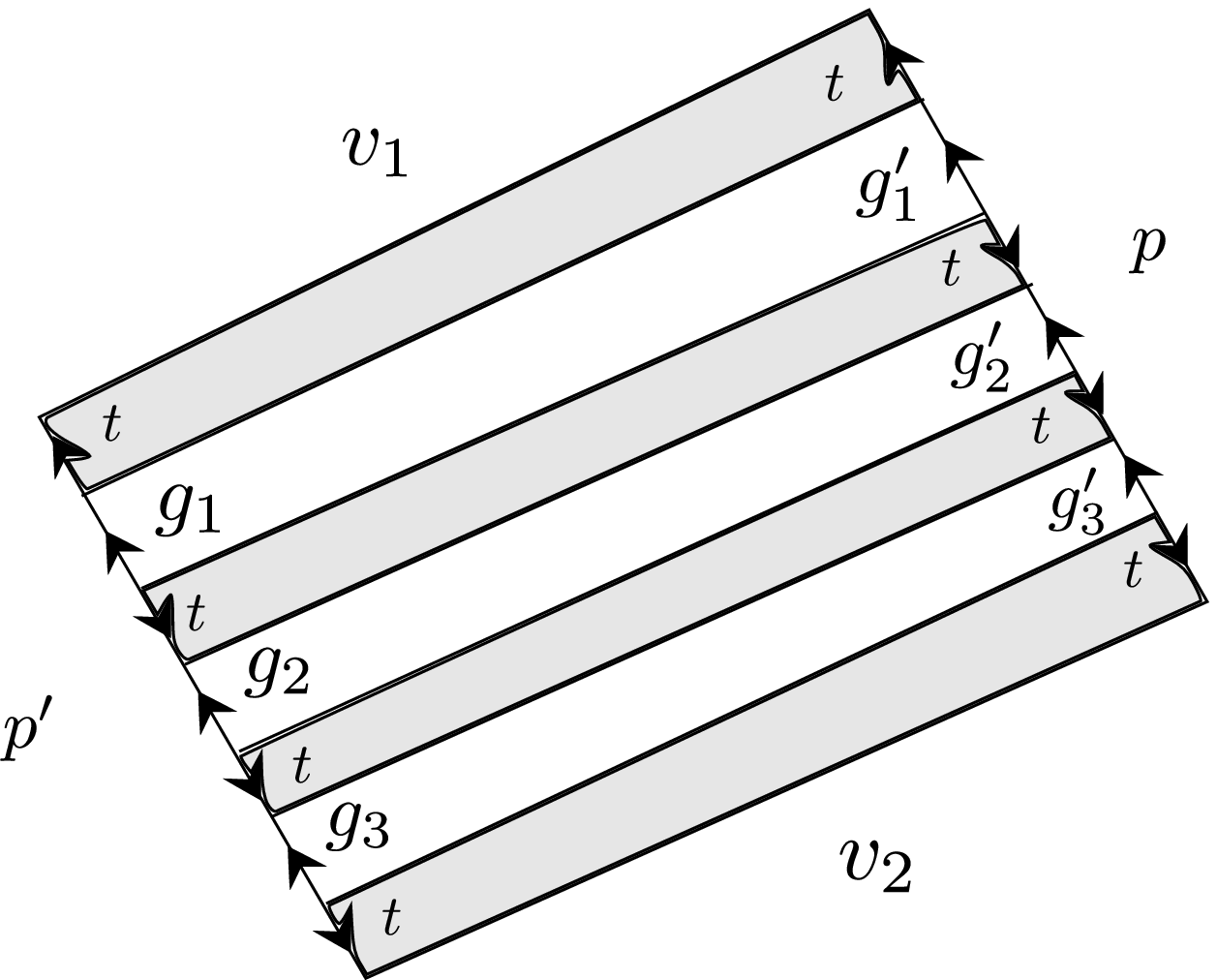}
\caption{}
\label{fig:piece}
\end{figure}

By Lemma \ref{solitar} $g_i=g'_i$ in the Baumslag-Solitar group
and the $t$-bands connecting them have length $0$ (there is a very limited number of possibilities for $g_i$, namely, it could be $x^{\pm 1}, x^{\pm 3}, x^{\pm 5}, x^{\pm 7}$ or $y^{\pm E_{n-1}}$). Therefore if the length of the piece is greater than $1$ then $v_1 = 1 = v_2 $ (in the notation of the Definition \ref{piece}).

Assume $p$ contains any of $x^{\pm 3}, x^{\pm 5}, x^{\pm 7}$. Condition (5) of Definition \ref{piece} and the fact that $v_1=1=v_2$ then implies that the only possibility is for $p$ and $p'$ to be from the relator and its inverse. We compare the relator, say, around $x^5$:

\begin{displaymath}
... ,t, x , t^{-1} , y^{E_{n-1}} , t, x^5  , t^{-1} , y^{-E_{n-1}}, t^1, x^{-1}, t^{-1}, ...
\end{displaymath}
to the inverse of the relator around $x^5$
\begin{displaymath}
... ,t, x ^{-1}, t^{-1} , y^{E_{n-1}} , t, x^5  , t^{-1} , y^{-E_{n-1}}, t^1, x, t^{-1}, ...
\end{displaymath}

Therefore, we see that if $p$ contains $x^5$ then $l_t(p) \leq 4$. We can proceed similarly for pieces containing $x^{\pm 3}, x^{- 5}, x^{\pm 7}$. If $p$ does not contain any of the $x^{\pm 3}, x^{\pm 5}, x^{\pm 7}$, then its length is similarly at most $4$ which is the length between these powers of $x$ except for $x^{-7}$ to $x^3$, where there is an extra $t^{-1}$ letter, which makes the pieces even shorter. The total length is $4*6+1$, so C'($\frac{1}{6}$,$E_n$) holds.

\end {proof}

Now we can prove Theorem \ref{second}.

\begin {proof}
Suppose for the sake of contradiction $x = \prod\limits_{i=1}^{N} g_i u_i^ {\pm 1} g_i^{-1}$, where the equality is in the free group, $N \leq E_{n-1}$ and $u_i$ are the relators from $\mu_n$. We want to rewrite this equality using $t^{-1}u_{n,1}$ instead of $t^{-1}u_n$.\\

We need to apply the relations of $\mu_0$ $2E_{n-1}$ times
to convert $t^{-1} y^{-E_{n-2}} x y^{E_{n-2}} t$ to $y^{E_{n-1}}$.
 Therefore, we need $24 \cdot 2E_{n-1}$ applications of the relations
to convert $u_{n,2}$ to $u_{n,1}$.
Similarly, we need $48 \cdot 2E_{n-2}$ applications
 to convert $u_{n,3}$ to $u_{n,2}$, etc.
 Since $E_{n-2}$ + $2E_{n-3} + 4E_{n-4}+...$ do not add up to more than $E_{n-1}$, we see that
we do need more than $96 \cdot E_{n-1}$ applications of the realtions
to convert $t^{-1}u_n$ to $t^{-1}u_{n,1}$.\\

There are at most $N$ such relators in the product. Therefore,
we need at most $96 \cdot E_{n-1} \cdot E_{n-1}$ relations of $\mu_0$
to convert all of them.
Since for $n>5$, $E_n > 96 \cdot E_{n-1} \cdot E_{n-1}$, we have $x = \prod\limits_{i=1}^{N'} g'_i (u'_i )^{\pm 1} (g'_i)^{-1}$, where the equality is in the free group, $N' \leq E_{n}$ and $u'_i$ are either from $\mu_0$ or are $t^{-1}u_{n,1}$. The theorem now follows from Lemma \ref{c6} and Theorem \ref {small}.
\end {proof}

We prove the main result of this section. Theorem \ref{main}.

\begin {proof}
Consider a sequence of presentations $\mu_i = \mu_i^{(0)}, \mu_i^{(1)}, \mu_i^{(2)}...$, where the last presentation is the empty presentation, and each step is an elementary Tietze transformation.
For some $n$ $\mu_i^{(n+1)}$ will be obtained from $\mu_i^{(n)}$ by applying $Op_5^{-1}$ to kill the generator $x$. $Area_{\mu_i^{(n)} } (x) = 1$, $Area_{\mu_i^{(0)} }(x) > E_{n-1}$ by Theorem \ref{second}. ${Op}_1$, ${Op}_1^{-1}$, ${Op}_2$, ${Op}_3$, ${Op}_5$, ${Op}_5^{-1}$ do not change the area of the word $x$. ${Op}_4$ can reduce the area by a factor
that cannot exceed $2$ for the following reason: Suppose $\mu_i^{(k+1)}  = ( \{x_1,...,x_r\} , \{a_1,...,a_{i-1}, a_ia_j, a_{i+1}, ..., a_p \} )$ is obtained from $\mu_i^{(k)} = ( \{x_1,...,x_r\} , \{a_1,...,a_{i-1}, a_i, a_{i+1}, ..., a_p \} )$ by applying ${Op}_4$. If it takes $K$ relators to show that $x=1$ in $\mu_i^{(k+1)}$, then in the worst case scenario all of them are $ a_ia_j$, and we will need $2K$ relators in $\mu_i^{(k)}$ to show that $x=1$. Therefore, $Area_{\mu_i^{(k)} }(x) \geq \frac {Area_{\mu_i^{(k+1)} }(x)}{2}$. So, we have $n > E_{n-2}$.

\end {proof}

To get from $\mu_i$ the $2$ relator presentations mentioned in the introduction we need to apply Tietze transformations to eliminate $y$ ($y = x^{t}$), the number of transformations needed grows linearly with the length of the relator $a_i$. The last claim from the introduction left to prove is that area of $t$ is also large. But that is clear because the area of $x$ can not be much larger than the area of $t$.\\

{\bf . Acknowledements.} I would like to thank my Ph.D. advisor Alexander
Nabutosky for introdicing me to the problems studied in this paper
and for numerous useful discussions.
This research had been partially supported
from his NSERC Accelerator and Discovery grants.

\end{document}